\documentclass[a4paper,reqno,12pt]{amsart}
\usepackage{verbatim}
\usepackage{amssymb,amsmath,amsthm}
\usepackage{amsfonts}
\usepackage{array}

\newtheorem{theorem}{Theorem}[section]
\newtheorem{lemma}{Lemma}[section]

\newtheorem{definition}[theorem]{Definition}
\theoremstyle{remark}
\newtheorem{remark}[theorem]{\bf Remark}

\begin{document}

\setcounter{page}{1}

\title[]{Multiple harmonic sums and multiple harmonic star sums are (nearly) never  integers}

\author{Kh.~Hessami Pilehrood}
\address{The Fields Institute for Research in Mathematical Sciences, 222 College Street, Toronto, ON M5T 3J1, Canada}
\email{hessamik@gmail.com}

\author{T.~Hessami Pilehrood}
\address{The Fields Institute for Research in Mathematical Sciences, 222 College Street, Toronto, ON M5T 3J1, Canada}
\email{hessamit@gmail.com}

\author{R.~Tauraso}
\address{Dipartimento di Matematica, 
Universit\`a di Roma ``Tor Vergata'', 
via della Ricerca Scientifica, 
00133 Roma, Italy}
\email{tauraso@mat.uniroma2.it}


\date{}

\begin{abstract}
It is  well known that the harmonic sum $H_n(1)=\sum_{k=1}^n\frac{1}{k}$ is never an integer for $n>1$. In 1946, Erd\H{o}s and Niven
proved that the nested multiple harmonic sum $H_n(\{1\}^r)=\sum_{1\le k_1<\dots<k_r\le n}\frac{1}{k_1\cdots k_r}$ can take integer values only
for a finite number of positive integers $n$. In 2012, Chen and Tang refined this result by showing that $H_n(\{1\}^r)$ is an integer
only for $(n,r)=(1,1)$ and $(n,r)=(3,2)$. In this paper, we consider the integrality problem for arbitrary multiple harmonic and multiple
harmonic star sums and show that none of these sums  is an integer with some natural exceptions like those mentioned above.
\end{abstract}

\maketitle

\section{Introduction} A well known result of elementary number theory is that even though the partial sum of the harmonic series $\sum_{k=1}^n 1/k$ increases to infinity, it is never an integer for $n>1$. Apparently the first published proof goes back to Leopold Theisinger in 1915 \cite{T}, and, since then, it has been
proposed as a challenging problem in several textbooks;
among all we mention  \cite[p.16, Problem 30]{IR}, \cite[p.33, Problem 37]{N}, \cite[p.153, Problem 250]{PS}.

In 1946, Erd\H{o}s and Niven \cite{EN} proved a stronger statement: there is only a finite number of integers $n$ for which there is a positive integer $r\leq n$ such that
the $r$-th elementary symmetric function of $1,1/2,\dots,1/n$, that is
$$
\sum_{1\le k_1<\cdots<k_r\le n}\frac{1}{k_1\cdots k_r},
$$
is an integer. In 2012, Chen and Tang \cite{CT} refined  this result and succeeded to show that the above sum is not an integer with the only two exceptions:  either $n=r=1$ or $n=3$ and $r=2$.
Recently, this theme has been further developed
by investigating the case when the variables of the elementary symmetric functions are $1/f(1),1/f(2),\dots,1/f(n)$ with
$f(x)$ being a polynomial of nonnegative integer coefficients:
see \cite{HW1} and \cite{HW2} for $f(x)=ax+b$ and see \cite{LHQW}
for $f$ of degree at least two.

In this paper, we  consider the integrality problem for sums which are not necessarily symmetric with respect to their variables.
For an $r$-tuple of positive integers
${\bf s}=(s_1, \ldots, s_r)$ and an integer $n\geq r$, we define two classes of multiple harmonic sums: the {\sl ordinary multiple harmonic sum} (MHS)
\begin{equation}
\label{f1}
H_n(s_1,\ldots,s_r)=
\sum_{1\le k_1<\cdots<k_r\le n}\frac{1}{k_1^{s_1}\cdots k_r^{s_r}},
\end{equation}
and the {\sl star} version
(MHS-star also denoted by $S$ in the literature)
\begin{equation}
\label{f2}
H^{\star}_n(s_1,\ldots,s_r)=\sum_{1\le k_1\le\cdots\le k_r\le n}\frac{1}{k_1^{s_1}\cdots k_r^{s_r}}.
\end{equation}
The number $l({\bf s}):=r$ is called the length and $|{\bf s}|:=\sum_{j=1}^rs_j$ is the weight of the multiple harmonic sum.
Note that $H_n(\{m\}^r)$ is the $r$-th elementary symmetric function of $1/f(1),1/f(2),\dots,1/f(n)$ with $f(x)=x^m$. The  multiple sums
\eqref{f1} and \eqref{f2} are of a certain interest because by taking the limit as $n$ goes to $\infty$ when $s_r>1$ (otherwise the infinite sums diverge)
we get the so-called {\sl multiple zeta value}
and the {\sl multiple zeta star value},
$$\lim_{n\to \infty} H_n(s_1,\ldots,s_r)=\zeta(s_1,\ldots,s_r)
\quad\mbox{and}\quad
\lim_{n\to \infty} H^{\star}_n(s_1,\ldots,s_r)=\zeta^{\star}(s_1,\ldots,s_r),
$$
respectively.
Note that the integrality of the MHS-star is quite simple to study.

\begin{theorem}\label{Sn}
Let $n\ge r$. Then
$H^{\star}_n(s_1,\ldots,s_r)$ is never an integer with the exception of $H^{\star}_1(s_1)=1$.
\end{theorem}
\begin{proof} If $n>1$, then by Bertrand's postulate,
there is at least a prime $p$ such that $n/2<p\leq n$. Then $p\leq n<2p$ and
$$H^{\star}({\bf s})=
\underset{\exists i\; :\; k_i\not=p}
{\sum_{1\le k_1\le\cdots\le k_r\le n}}\frac{1}{k_1^{s_1}\cdots k_r^{s_r}}
+\underset{\forall i\,,\, k_i=p}
{\sum_{1\le k_1\le\cdots\le k_r\le n}}\frac{1}{k_1^{s_1}\cdots k_r^{s_r}}
=\frac{a}{b p^t}+\frac{1}{p^{|{\bf s}|}}$$
where $\gcd(b,p)=1$ and $t<|{\bf s}|$. Assume that $H^{\star}({\bf s})=m\in\mathbb{N}^+$. Then $bmp^{|{\bf s}|}=ap^{|{\bf s}|-t}+b$, which is a contradiction because $p$ divides the l.h.s.\ and $p$ does not divide the r.h.s..
\end{proof}

On the other hand, the case of the ordinary MHS is much more intricate. Our result is given below whereas the entire next section is dedicated to its proof.

\begin{theorem}\label{Hn}
Let $n\ge r$. Then
$H_n(s_1,\ldots,s_r)$ is never an integer with the exceptions of $H_1(s_1)=1$ and $H_3(1,1)=1$.
\end{theorem}

\noindent Throughout the paper, all the numerical computations were performed by using Maple\texttrademark.

\section{Proof of Theorem \ref{Hn}}

The main tool in the proof of Theorem \ref{Sn} is  Bertrand's  postulate. Unfortunately, such result is not strong enough to imply Theorem \ref{Hn}. The statement described in the following remark will replace it. Notice that the same argument had been used by Erd\H{o}s and Niven in \cite{EN} and successively taken up in
\cite{CT}, \cite{HW1} and \cite{HW2}.

\begin{remark} \label{R3}
Let $1\leq r\leq n$. If there is a prime $p\in (\frac{n}{r+1},\frac{n}{r}]$
and $r<p$, then $H_n({\bf s})$ is not an integer when $l({\bf s})=r$.
Indeed, since $1<p<2p<\dots <rp\leq n<(r+1)p$, it follows that
$$H_n({\bf s})=
\underset{\exists i\; :\; p \nmid k_i}
{\sum_{1\le k_1<\cdots< k_r\le n}}\frac{1}{k_1^{s_1}\cdots k_r^{s_r}}
+\underset{\forall i\,,\, p \mid k_i}
{\sum_{1\le k_1<\cdots< k_r\le n}}\frac{1}{k_1^{s_1}\cdots k_r^{s_r}}
=\frac{a}{b p^t}+\frac{1}{c p^{|{\bf s}|}}$$
where $\gcd(b,p)=\gcd(c,p)=1$ and $t<|{\bf s}|$. Assume that $H_n({\bf s})=m\in\mathbb{N}^+$. Then $bcmp^{|{\bf s}|}=cap^{|{\bf s}|-t}+b$, which is a contradiction because $p$ divides the l.h.s. and $p$ does not divide the r.h.s..
\end{remark}

For any integer $r\geq 1$, let
$$A_r=\bigcup_{p\in \mathbb{P}}[rp,(r+1)p)$$
where $\mathbb{P}$ is the set of primes. Note that by Bertrand's postulate, $A_1=[2,+\infty)$. The crucial property of the set $A_r$ is that $n\in A_r$ if and only if there exists a prime $p$ such that
$p\in (\frac{n}{r+1},\frac{n}{r}]$.
The next lemma is a variation of \cite[Lemma 2.4]{HW1}.

\begin{lemma}\label{L1}  For any positive integer $r$, $A_r$ is cofinite, i.e.,
$\mathbb{N}\setminus A_r$ is finite.
Let $m_r=\max(\mathbb{N}\setminus A_r)+1$. Then the first few values of $m_r$ are as follows:
\vspace*{3mm}
\begin{center}\rm
\setlength{\extrarowheight}{3pt}
\begin{tabular}{|c|c|c|c|c|c|c|c|c|c|c|c|c|c|}
\hline
$r$  & 1 & 2-3  & 4-5  & 6-7   & 8   & 9-13   & 14-16   & 17   & 18-20    & 21-22   & 23-29   & 30-39 & 40-69  \\
\hline
$m_r/r$ & 2 & 11 & 29 & 37 & 53 & 127 & 149 & 211 & 223 & 307 & 331 & 541 & 1361 \\
\hline
\end{tabular}
\end{center}
\vspace*{3mm}
and for  any positive integer $r\geq 24$, $m_r\leq (r+1)\exp(\sqrt{1.4\,r})$.
\end{lemma}
\begin{proof} Let $r\geq 1$ and let $n$ be such that $x:=n/(r+1)\geq \max(\exp(\sqrt{1.4\,r}),3275)$ .
Then $\ln^2(x)\geq 1.4\,r$ implies
$$\frac{n}{r}=x+\frac{x}{r}\geq x+\frac{1.4\,x}{\ln^2(x)}\geq x+\frac{x}{2\ln^2(x)}.$$
Moreover, by \cite[Theorem 1]{D}, there is a prime $p$ such that
$$p\in \left(x, x+\frac{x}{2\ln^2(x)}\right]\subset \left(\frac{n}{r+1},\frac{n}{r}\right]$$
which implies that $n\in A_r$. This proves that $\mathbb{N}\setminus A_r$  is finite.

Note that if $r\geq 47$, then $\exp(\sqrt{1.4\,r})\geq 3275$.
On the other hand, if
$24\leq r\leq 46$ and $n\in [(r+1)\exp(\sqrt{1.4\,r}),(r+1)3275)$, one can verify directly that $n\in A_r$.
\end{proof}

In order to compare  values of multiple harmonic sums of the same length, the following definition and lemma will be useful.

\begin{definition}
Let ${\bf s}=(s_1,\ldots,s_r)$ and ${\bf t}=(t_1,\ldots,t_r)$ be two $r$-tuples of positive integers.
We say that ${\bf s}\ge {\bf t}$, if $w({\bf s})\ge w({\bf t})$, i.e., if $s_1+\dots+s_r\ge t_1+\dots+t_r$,
and $s_1\le t_1, \ldots, s_l\le t_l, s_{l+1}\ge t_{l+1}, \ldots, s_r\ge t_r$ for some $0\le l\le r-1$.
In particular, ${\bf s}\ge {\bf t}$ if $s_i\ge t_i$ for $i=1,\dots,r$.
\end{definition}

\begin{lemma}\label{L2}
Let ${\bf s}=(s_1,\ldots,s_r)$ and ${\bf t}=(t_1,\ldots,t_r)$ be two $r$-tuples of positive integers, and
${\bf s}\ge {\bf t}$. Then for any positive integer $n$,
$$H_n(s_1, \ldots, s_r)\le H_n(t_1, \ldots, t_r).$$
\end{lemma}
\begin{proof}
Let $0\le l\le r-1$ be such that
$s_1\le t_1, \ldots, s_l\le t_l$ and $s_{l+1}\ge t_{l+1}, \ldots, s_r\ge t_r$.
Now we compare corresponding terms
$$\frac{1}{k_1^{s_1}\cdots k_r^{s_r}} \quad \text{and} \quad \frac{1}{k_1^{t_1}\cdots k_r^{t_r}}, \qquad  \text{where} \quad
k_1<\dots<k_r,$$
of multiple sums $H_n({\bf s})$ and $H_n({\bf t})$.
Since $s_1+\dots+s_r\ge t_1+\dots+t_r$, we have
\begin{equation*}
\begin{split}
k_1^{t_1-s_1}k_2^{t_2-s_2}\cdots k_l^{t_l-s_l}&\le k_l^{t_1+\dots+t_l-(s_1+\dots+s_l)} \\
&=k_l^{(s_{l+1}-t_{l+1})+\dots+(s_r-t_r)}\cdot
k_l^{t_1+\dots+t_r-(s_1+\dots+s_r)} \\
&\le k_{l+1}^{s_{l+1}-t_{l+1}}\cdots k_r^{s_r-t_r}
\end{split}
\end{equation*}
and therefore,
$$
\frac{1}{k_1^{s_1}\cdots k_r^{s_r}}\le\frac{1}{k_1^{t_1}\cdots k_r^{t_r}},
$$
which implies $H_n({\bf s})\le H_n({\bf t})$, as required.
\end{proof}

\begin{remark} \label{R2} Notice that
$$H_n(1)=1+\sum_{k=2}^n\frac{1}{k}\leq 1+\int_1^n \frac{dx}{x}=\ln(n)+1.$$
Hence, if $1\leq r\leq n$ then, by the previous lemma,
$$H_n(s_1,\ldots,s_r)\leq H_n(\{1\}^r)\leq\frac{\left(H_n(1)\right)^r}{r!}\leq\frac{\left(\ln(n)+1\right)^r}{r!}$$
where the second inequality holds because each term of $H_n(\{1\}^r)$
is contained $r!$ times in the expansion of $(H_n(1))^r$.
\end{remark}

Let $p$ be a prime and let $\nu_p(q)$ be the {\sl $p$-adic order} of the rational number $q$, that is, if $a,b$ are coprime with $p$ and $n\in\mathbb{Z}$, then $\nu_p(ap^n/b)=n$. It is known that the $p$-adic order satisfies
the inequality
$$\nu_p(a+b)\geq \min(\nu_p(a),\nu_p(b))$$
where the equality holds if $\nu_p(a) \neq \nu_p(b)$.
The following lemma will be our basic tool to find an upper bound for the index $s_1$.

\begin{lemma} \label{L3}
Given $2\leq r\leq n$ and $(s_2,\dots, s_r)$, then there exists an integer $M$ (which depends on $(s_2,\dots, s_r)$ and $n$) such that $H_n(s_1,s_2,\ldots,s_r)$ is never an integer for any positive integer $s_1>M$.
\end{lemma}
\begin{proof} If $r=n$, then $H_n(s_1,s_2,\ldots,s_r)$ is trivially not an integer for all $s_1>0$.

\noindent If $r<n$, then we have that
$$H_n(s_1,\ldots,s_r)=\sum_{k=1}^{n-r+1}\frac{c_k}{k^{s_1}}\qquad\mbox{where}\qquad
c_k:=\sum_{k<k_2<\cdots<k_r\le n}\frac{1}{k_2^{s_2}\cdots k_r^{s_r}}>0.$$
Let $p$ be the largest prime in $[2,n-r+1]$. Then
$2p>n-r+1$ (otherwise by Bertrand's postulate there is a prime $q$ such that $p<q<2p\leq n-r+1$). Let
$$M:=\max\left(\nu_p(c_p),\nu_p(c_p)-\underset{k\not=p}{\min_{1\le k\le n-r+1}}(\nu_p(c_k))\right).$$
We will show that $\nu_p(H_n({\bf s}))<0$, which implies
that $H_n({\bf s})$ is not an integer.
Assume that $s_1>M$, then
$$\mbox{i) }
s_1>\nu_p(c_p)\qquad\qquad\mbox{and}\qquad\qquad
\mbox{ii) }  s_1>\nu_p(c_p)-\underset{k\not=p}{\min_{1\le k\le n-r+1}}(\nu_p(c_k)).$$
Now
$$H_n({\bf s})=\frac{c_p}{p^{s_1}}+
\underset{k\not=p}
{\sum_{1\le k\le n-r+1}}\frac{c_k}{k^{s_1}}.$$
By ii), we have that
$$\nu_p\left(\underset{k\not=p}
{\sum_{1\le k\le n-r+1}}\frac{c_k}{k^{s_1}}\right)
\geq \underset{k\not=p}
{\min_{1\le k\le n-r+1}} \nu_p\left(\frac{c_k}{k^{s_1}}\right)
=\underset{k\not=p}
{\min_{1\le k\le n-r+1}} \nu_p\left(c_k\right)
>\nu_p(c_p)-s_1=\nu_p\left(\frac{c_p}{p^{s_1}}\right).$$
and by i),
$$\nu_p\left(\frac{c_p}{p^{s_1}}\right)=\nu_p(c_p)-s_1<0.$$
Therefore
$$\nu_p\left(H_n({\bf s})\right)=
\min\left(
\nu_p\left(
\underset{k\not=p}
{\sum_{1\le k\le n-r+1}}
\frac{c_k}{k^{s_1}}
\right),
\nu_p\left(\frac{c_p}{p^{s_1}}\right)
\right)
=\nu_p\left(\frac{c_p}{p^{s_1}}\right)<0.
$$
\end{proof}

\begin{remark} \label{R1}
If, in addition to the inequalities $2\leq p\leq n-r+1$, the prime $p$ satisfies $n/2<p$, then
$$\nu_p(c_k)\geq \min_{k<k_2<\cdots<k_r\le n}{\nu_p\left(\frac{1}{k_2^{s_2}\cdots k_r^{s_r}}\right)}
=\begin{cases}
-\max_{2\leq i\leq r}{s_i}, &\text{if $1\leq k<p$;} \\
0, &\text{if $p\leq k\leq n-r+1$.}
\end{cases}$$
Therefore
$$\underset{k\not=p}{\min_{1\le k\le n-r+1}}(\nu_p(c_k))\geq -\max_{2\leq i\leq r}{s_i}.$$
This means that
$$M':=\nu_p(c_p)+\max_{2\leq i\leq r}{s_i}\geq M.$$
When the assumption $n/2<p$ is met, we will prefer to use the bound $M'$ to $M$ because the bound $M'$ is computationally easier to determine.
\end{remark}

\begin{proof}[{\bf Proof of Theorem \ref{Hn}}] Since $H_n(s_1)=H^{\star}_n(s_1)$, by Theorem \ref{Sn}, the statement holds for $r=1$.
If $e(\ln(n)+1)\leq r$, then
by Remark \ref{R2} and the fact that $e^r>r^r/r!$, we have
$$H_{n}({\bf s})\leq\frac{(\ln(n)+1)^r}{r!}
<\left(\frac{\ln(n)+1}{r/e}\right)^r\leq 1,
$$
which implies that $H_n({\bf s})$ can not be an integer.

\noindent Assume that $2\leq r<e(\ln(n)+1)$. Then $\exp(r/e-1)<n$ and, since it can be  verified that
$$(r+1)\exp(\sqrt{1.4\,r})\leq \exp(r/e-1) \quad \mbox{for $r\geq 30$,}$$
it follows that $n\in A_r$ and we are done as soon as we use Remark \ref{R3}.

\noindent Remark \ref{R3} can be applied successfully even
when $2\leq r\leq 29$ and $n\geq 9599$ because
$$m_r\leq 331 r\leq 331\cdot 29=9599\leq n.$$

\noindent Hence the cases  remained to  consider are $2\leq r\leq 29$ and $r\leq n< m_r$ (with $n\not\in A_r$).

\noindent For $r=25,26,27,28,29$, by Remark \ref{R2},
$$H_{n}({\bf s})<H_{m_r}({\bf s})\leq
H_{m_{r}}(\{1\}^{r})\leq\frac{(\ln(m_r)+1)^r}{r!}<1$$
where the last inequality can be easily verified numerically.

\noindent For $r=19,20,21,22,23,24$, by Lemma \ref{L2},
$$H_{n}({\bf s})<H_{m_r}({\bf s})\leq
H_{m_{r}}(\{1\}^{r})$$
and the non-integrality of $H_{n}({\bf s})$ is implied by the following evaluations:
\begin{center}\rm
\begin{tabular}l
$H_{m_{24}}(\{1\}^{24})<0.025084028<1$,\\
$H_{m_{23}}(\{1\}^{23})<0.068740285<1$,\\
$H_{m_{22}}(\{1\}^{22})<0.145564965<1$,\\
$H_{m_{21}}(\{1\}^{21})<0.369820580<1$,\\
$H_{m_{20}}(\{1\}^{20})<0.379560254<1$,\\
$H_{m_{19}}(\{1\}^{19})<0.916202538<1$.\\
\end{tabular}
\end{center}
\noindent The strategy to handle the cases where
$2\leq r\leq 18$ is fairly more complicated because
$H_{m_{r}}(\{1\}^{r})>1$.
The analysis is based on the numerical values presented in Tables
\ref{Tab1}, \ref{Tab2}, and \ref{Tab3}.
Here we give a detailed explanation of how the data in such tables are calculated and used for $r=5$. The other cases can be treated in a similar way.  What turns out at the end is that the only exception for $r\geq 2$ is $H_3(1,1)=1$.

\noindent We first determine the {\sl optimal set} of length $5$,
that is, a set of $5$-tuples such that the multiple harmonic sums with $n=m_5=145$ are the largest sums less than $1$ with small weights
(columns $2$ and $3$):
\begin{center}\rm
\begin{tabular}l
$H_{m_{5}}(\{1\}^4,2)<0.502399297<1$,\\
$H_{m_{5}}(1,2,2,1,1)<0.851108767<1$,\\
$H_{m_{5}}(1,4,\{1\}^3)<0.883176754<1$.
\end{tabular}
\end{center}
Then, thanks to Lemma \ref{L2}, the size of the set of multiple harmonic sums less than $1$ (and therefore not integral) can be extended.

\noindent Let $5\le n<m_5$. If $s_2\ge 4$, then
$$
H_n(s_1,s_2,s_3,s_4,s_5)\le H_n(1,4,\{1\}^{3})<
H_{m_5}(1,4,\{1\}^{3})<1.
$$
If $s_2\in\{2,3\}$ and there is $s_j\ge 2$ with $3\le j\le 5$, then
$$
H_n(s_1,s_2,s_3,s_4,s_5)\le H_n(1,2,s_3,s_4,s_5)
\le H_n(1,2,2,1,1)< H_{m_5}(1,2,2,1,1)<1.
$$
If $s_2=1$ and $s_3\ge 3$, then
$$
H_n(s_1,1,s_3,s_4,s_5)\le H_n(1,1,3,1,1)\le H_n(1,2,2,1,1)
< H_{m_5}(1,2,2,1,1)<1.
$$
If $s_2=1$, $s_3=2$, and $s_4\ge 2$ or $s_5\ge 2$, then
\begin{align*}
H_n(s_1,1,2,s_4,s_5)&\le H_n(1,1,2,s_4,s_5)\le H_n(1,1,2,2,1) \\
&\le H_n(1,2,2,1,1)< H_{m_5}(1,2,2,1,1)<1.
\end{align*}
If $s_2=1$, $s_3=1$, and $s_4\ge 3$, then
\begin{align*}
H_n(s_1,1,1,s_4,s_5)\le H_n(\{1\}^3,3,1)\le H_n(1,2,2,1,1)
< H_{m_5}(1,2,2,1,1)<1.
\end{align*}
If $s_2=1$, $s_3=1$, $s_4=2$, and $s_5\ge 2$, then
\begin{align*}
H_n(s_1,1,1,2,s_5)&\le H_n(\{1\}^3,2,s_5)
\le H_n(\{1\}^3,2,2) \\
&\le H_n(1,2,2,\{1\}^{2})< H_{m_5}(1,2,2,\{1\}^{2})<1.
\end{align*}
If $s_2=s_3=s_4=1$ and $s_5\ge 2$, then
\begin{align*}
H_n(s_1,\{1\}^3,s_5)\le H_n(\{1\}^4,s_5)\le H_n(\{1\}^4,2)
< H_{m_5}(\{1\}^4,2)<1.
\end{align*}
The set of $5$-tuples of positive integers which are excluded by the analysis above is what we call the {\sl exclusion set} (column $4$):
$$(s_1,3,\{1\}^3),\;\;
(s_1,2,\{1\}^3),\;\;
(s_1,1,2,1,1),\;\;
(s_1,1,1,2,1),\;\;
(s_1,\{1\}^4)$$
with $s_1\geq 1$.
By using Lemma \ref{L3} and Remark \ref{R1}, we are able to give an upper bound for $s_1$ (column $5$) and therefore to reduce the size of the exclusion set to a finite number.
Finally, it suffices to compute the rational number $H_n({\bf s})$ for
$5 \leq n <m_5$ and for ${\bf s}$ in the exclusion set with the upper bound for $s_1$ established in the last column. For $r=5$, none of them is an integer.
\end{proof}

\noindent\textbf{Acknowledgement.} Kh.~Hessami Pilehrood and T.~Hessami Pilehrood gratefully  acknowledge support from the Fields Institute Research Immersion Fellowships.

\begin{table}
\setlength{\extrarowheight}{1pt}
\caption{Optimal sets, exclusion sets and upper bounds for
$2\leq r\leq 8$} \label{Tab1}
\begin{center}\rm
\begin{tabular}{|c|c|c|c|c|}
\hline
\text{Length}      & \text{Optimal set of} &\text{Upper bound for}
&\text{Exclusion set}& \text{Upper bound}   \\
$r$& $(s_1,s_2,\dots,s_r)$ &$H_{m_r}(s_1,s_2,\dots,s_r)$ &\text{with} $s_1\ge 1$
&\text{for} $s_1$\\
\hline
2&$(1,2)$ &$<0.994099321$&$(s_1,1)$&$\leq 3$\\
\hline
3&$(1,1,2)$&$<0.706260681$  &$(s_1,2,1)$&$\leq 3$\\
 & $(1,3,1)$&$<0.589038111$ &$(s_1,1,1)$&$\leq 3$\\
\hline
 &$(1,1,1,2)$&$<0.684621751$  &$(s_1,3,1,1)$&$\leq 4$\\
4&$(1,2,2,1)$&$<0.537884954$  &$(s_1,2,1,1)$&$\leq 3$\\
 &$(1,4,1,1)$& $<0.627965284$  &$(s_1,1,2,1)$&$\leq 3$\\
 &                 &   &$(s_1,1,1,1)$&$\leq 3$\\
 \hline
 & &  & $(s_1,3,\{1\}^3)$&$\leq 4$\\
 &$(\{1\}^4,2)$ &$<0.502399297$  &$(s_1,2,\{1\}^3)$&$\leq 3$\\
 5&$(1,2,2,1,1)$&$<0.851108767$  &$(s_1,1,2,1,1)$&$\leq 3$\\
 &$(1,4,\{1\}^3)$&  $<0.883176754$  &$(s_1,1,1,2,1)$&$\leq 3$\\
 &   &   &$(s_1,\{1\}^4)$&$\leq 3$\\
 \hline
 & & &$(s_1,4,\{1\}^4)$&$\leq 5$\\
 &$(\{1\}^4,2,1)$  &$<0.861061229$ &$(s_1,3,\{1\}^4)$&$\leq 4$\\
&$(1,1,3,\{1\}^3)$&$<0.675761881$  &$(s_1,2,2,\{1\}^3)$&$\leq 3$\\
6& $(1,2,1,2,1,1)$& $<0.571129690$ &$(s_1,2,\{1\}^4)$&$\leq 3$\\
 &$(1,3,2,\{1\}^3)$&$<0.506341468$  &$(s_1,1,2,\{1\}^3)$&$\leq 3$\\
 & $(1,5,\{1\}^4)$&$<0.561379826$ &$(s_1,1,1,2,1,1)$&$\leq 3$\\
 &  &  &$ (s_1,\{1\}^5)$&$\leq 2$\\
 \hline
 & & &$(s_1,4,\{1\}^5)$&$\leq 5$\\
 &$(\{1\}^4,2,1,1)$&$<0.998935309$ &$(s_1,3,\{1\}^5)$&$\leq 4$\\
 &$(1,1,3,\{1\}^4)$&$<0.712465109$  &$(s_1,2,2,\{1\}^4)$&$\leq 3$\\
 7& $(1,2,1,2,\{1\}^3)$& $<0.637523826$ &$(s_1,2,\{1\}^5)$&$\leq 3$\\
 &$(1,3,2,\{1\}^4)$& $<0.531072055$ &$(s_1,1,2,\{1\}^4)$&$\leq 3$\\
 & $(1,5,\{1\}^5)$ &$<0.553285019$  &$(s_1,1,1,2,\{1\}^3)$&$\leq 4$\\
 &  &  &$ (s_1,\{1\}^6)$&$\leq 2$\\
 \hline
 & & &$(s_1,4,\{1\}^6)$&$\leq 5$\\
 &$(\{1\}^5,2,1,1)$&$<0.756114380$ &$(s_1,3,\{1\}^6)$&$\leq 4$\\
 &$(1,1,3,\{1\}^5)$&$<0.985465806$ &$(s_1,2,2,\{1\}^5)$&$\leq 5$\\
 8& $(1,2,1,2,\{1\}^4)$&$<0.896426675$  &$(s_1,2,\{1\}^6)$&$\leq 3$\\
 &$(1,3,2,\{1\}^5)$&$<0.733626794$  &$(s_1,1,2,\{1\}^5)$&$\leq 4$\\
 & $(1,5,\{1\}^6)$ &$<0.750552945$ &$(s_1,1,1,2,\{1\}^4)$&$\leq 4$\\
 &   &  &$(s_1,\{1\}^3,2,\{1\}^3)$&$\leq 2$\\
 &  &  &$ (s_1,\{1\}^6)$&$\leq 2$\\
 \hline
\end{tabular}
\end{center}
\end{table}
\begin{table}
\setlength{\extrarowheight}{1pt}
\caption{Optimal sets, exclusion sets and upper bounds for
$9\leq r\leq 11$} \label{Tab2}
\begin{center}\rm
\begin{tabular}{|c|c|c|c|c|}
\hline
\text{Length}      & \text{Optimal set of} &\text{Upper bound for}
&\text{Exclusion set}& \text{Upper bound}   \\
$r$& $(s_1,s_2,\dots,s_r)$ &$H_{m_r}(s_1,s_2,\dots,s_r)$ &\text{with} $s_1\ge 1$
&\text{for} $s_1$\\
\hline
 &  &  &$(s_1,5,\{1\}^7)$&$\leq 7$\\
 &  &  &$(s_1,4,\{1\}^7)$&$\leq 5$\\
 &  &  &$(s_1,3,2,\{1\}^6)$&$\leq 6$\\
 &$(\{1\}^6,2,1,1)$&$<0.925976971$&$(s_1,3,\{1\}^7)$&$\leq 4$\\
 &$(1,2,\{1\}^3,2,\{1\}^3)$&$<0.570757133$   &$(s_1,2,2,\{1\}^6)$&$\leq 3$\\
 &$(1,1,2,1,2,\{1\}^4)$&$<0.591649570$  & $(s_1,2,1,2,\{1\}^5)$&$\leq 4$\\
9&$(1,1,1,3,\{1\}^5)$&$<0.707900100$   &$(s_1,2,1,1,2,\{1\}^4)$&$\leq 3$\\
  &$(1,3,1,2,\{1\}^5)$&$<0.884650814$    &$(s_1,2,\{1\}^7)$&$\leq 4$\\
  &$(1,2,3,\{1\}^6)$&$<0.979701671$       &$(s_1,1,3,\{1\}^6)$&$\leq 4$\\
  &$(1,4,2,\{1\}^6)$&$<0.814983421$      &$(s_1,1,2,2,\{1\}^5)$&$\leq 3$\\
  &$(1,6,\{1\}^7)$&    $<0.897735179$    &$(s_1,1,2,\{1\}^6)$&$\leq 3$\\
  & & &$(s_1,1,1,2,\{1\}^5)$&$\leq 3$  \\
  &  & &$(s_1,\{1\}^3,2,\{1\}^4)$&$\leq 3$  \\
  &  &  &$(s_1,\{1\}^4,2,\{1\}^3)$&$\leq 3$\\
  &  &  &$(s_1,\{1\}^8)$&$\leq 2$\\
  \hline
  &  &  &$(s_1,5,\{1\}^8)$&$\leq 6$\\
 &  &  &$(s_1,4,\{1\}^8)$&$\leq 6$\\
 &  &  &$(s_1,3,2,\{1\}^7)$&$\leq 4$\\
&$(\{1\}^6,2,\{1\}^3)$&$<0.938457721$&$(s_1,3,\{1\}^8)$&$\leq 4$\\
&$(1,2,\{1\}^3,2,\{1\}^4)$&$<0.561773422$ &$(s_1,2,2,\{1\}^7)$&$\leq 4$\\
&$(1,1,2,1,2,\{1\}^5)$&$<0.558322794$  & $(s_1,2,1,2,\{1\}^6)$&$\leq 4$\\
10&$(1,1,1,3,\{1\}^6)$&$<0.644468502$   &$(s_1,2,1,1,2,\{1\}^5)$&$\leq 3$\\
  &$(1,3,1,2,\{1\}^6)$&$<0.787882830$    &$(s_1,2,\{1\}^8)$&$\leq 3$\\
  &$(1,2,3,\{1\}^7)$&$<0.847077826$       &$(s_1,1,3,\{1\}^7)$&$\leq 4$\\
  &$(1,4,2,\{1\}^7)$&$<0.697660582$      &$(s_1,1,2,2,\{1\}^6)$&$\leq 3$\\
  &$(1,6,\{1\}^8)$&$<0.739316195$ &$(s_1,1,2,\{1\}^7)$&$\leq 4$\\
  & & &$(s_1,1,1,2,\{1\}^6)$&$\leq 3$  \\
  &  & &$(s_1,\{1\}^3,2,\{1\}^5)$&$\leq 4$  \\
  &  &  &$(s_1,\{1\}^4,2,\{1\}^4)$&$\leq 3$\\
  &  &  &$(s_1,\{1\}^9)$&$\leq 3$\\
  \hline
  &  &  &$(s_1,5,\{1\}^9)$&$\leq 8$\\
    & & &$(s_1,4,\{1\}^9)$&$\leq 5$\\
  &  &  &$(s_1,3,2,\{1\}^8)$&$\leq 6$\\
  &$(\{1\}^6,2,\{1\}^4)$&$<0.814836649$&$(s_1,3,1,\{1\}^8)$&$\leq 5$\\
  &$(1,2,1,1,2,\{1\}^6)$&$<0.841877884$  &$(s_1,2,2,\{1\}^8)$&$\leq 4$\\
  &$(1,1,2,2,\{1\}^7)$&$<0.836565448$  &$(s_1,2,1,2,\{1\}^7)$&$\leq 3$\\
11&$(1,3,1,2,\{1\}^7)$&$<0.621949675$  &$(s_1,2,1,1,\{1\}^7)$&$\leq 3$\\
 &$(1,2,3,\{1\}^8)$&$<0.653475111$   &$(s_1,1,3,\{1\}^8)$&$\leq 5$\\
 &$(1,4,2,\{1\}^8)$&$<0.533753072$   &$(s_1,1,2,\{1\}^8)$&$\leq 3$\\
 &$(1,6,\{1\}^9)$&$<0.548020075$&$(s_1,1,1,2,\{1\}^7)$&$\leq 3$\\
 &  &  &$(s_1,\{1\}^3,2,\{1\}^6)$&$\leq 3$\\
 &  &  &$(s_1,\{1\}^4,2,\{1\}^5)$&$\leq 3$ \\
 &  &  &$(s_1,\{1\}^{10})$&$\leq 3$\\
 \hline
\end{tabular}
\end{center}
\end{table}
\begin{table}
\setlength{\extrarowheight}{1pt}
\caption{Optimal sets, exclusion sets and upper bounds for
$12\leq r\leq 18$} \label{Tab3}
\begin{center}\rm
\begin{tabular}{|c|c|c|c|c|}
\hline
\text{Length}      & \text{Optimal set of} &\text{Upper bound for}
&\text{Exclusion set}& \text{Upper bound}   \\
$r$& $(s_1,s_2,\dots,s_r)$ &$H_{m_r}(s_1,s_2,\dots,s_r)$ &\text{with} $s_1\ge 1$
&\text{for} $s_1$\\
\hline
   & & &$(s_1,4,\{1\}^{10})$&$\leq 5$\\
  &  &  &$(s_1,3,\{1\}^{10})$&$\leq 4$\\
  &$(\{1\}^6,2,\{1\}^5)$&$<0.622525355$&$(s_1,2,2,\{1\}^9)$&$\leq 3$\\
 &$(1,2,1,1,2,\{1\}^7)$&$<0.611346842$&$(s_1,2,1,2,\{1\}^8)$&$\leq 4$\\
 12 &$(1,1,2,2,\{1\}^8)$&$<0.595133583$  &$(s_1,2,\{1\}^{10})$&$\leq 3$\\
&$(1,3,2,\{1\}^9)$&$<0.816806820$  &$(s_1,1,3,\{1\}^9)$&$\leq 4$\\
 &$(1,5,\{1\}^{10})$& $<0.779899998$  &$(s_1,1,2,\{1\}^{9})$&$\leq 4$\\
 & &    &$(s_1,1,1,2,\{1\}^8)$&$\leq 3$\\
 &  &  &$(s_1,\{1\}^3,2,\{1\}^7)$&$\leq 3$\\
 &  &  &$(s_1,\{1\}^4,2,\{1\}^6)$&$\leq 3$ \\
 &  &  &$(s_1,\{1\}^{11})$&$\leq 3$\\
 \hline
 & & &$(s_1,4,\{1\}^{11})$&$\leq 5$\\
 &$(\{1\}^5,2,\{1\}^7)$&$<0.684109082$   &$(s_1,3,\{1\}^{11})$&$\leq 4$\\
 &$(1,2,1,2,\{1\}^9)$&$<0.678679596$   &$(s_1,2,2,\{1\}^{10})$&$\leq 4$\\
13&$(1,1,3,\{1\}^{10})$&$<0.694738378$  &$(s_1,2,\{1\}^{11})$&$\leq 3$\\
 &$(1,3,2,\{1\}^{10})$&$<0.514442958$   &$(s_1,1,2,\{1\}^{10})$&$\leq 3$\\
 &$(1,5,\{1\}^{11})$&$<0.481478855$      &$(s_1,1,1,2,\{1\}^9)$&$\leq 3$\\
 &   &   &$(s_1,\{1\}^3,2,\{1\}^8)$&$\leq 3$\\
 &   &   &$(s_1,\{1\}^{12})$&$\leq 2$\\
 \hline
 & &  &$(s_1,3,\{1\}^{12})$&$\leq 4$\\
 &$(\{1\}^4,2,\{1\}^9)$&$<0.887408969$&$(s_1,2,\{1\}^{12})$&$\leq 3$\\
14& $(1,2,2,\{1\}^{11})$&$<0.951784321$  &$(s_1,1,2,\{1\}^{11})$&$\leq 3$\\
 &$(1,4,\{1\}^{12})$      & $<0.818231535$ &$(s_1,1,1,2,\{1\}^{10})$&$\leq 3$\\
 &  &  &$(s_1,\{1\}^{13})$&$\leq 3$\\
 \hline
 &$(\{1\}^3,2,\{1\}^{11})$& $<0.799176532$ &$(s_1,3,\{1\}^{13})$&$\leq 4$\\
15&$(1,2,2,\{1\}^{12})$&$<0.521690836$   &$(s_1,2,\{1\}^{13})$&$\leq 3$\\
&$(1,4,\{1\}^{13})$& $<0.443394022$   &$ (s_1,1,2,\{1\}^{12})$&$\leq 3$\\
  & & &$(s_1,\{1\}^{14})$&$\leq 2$\\
  \hline
16&$(1,1,2,\{1\}^{13})$&$<0.683053316$& $(s_1,2,\{1\}^{14})$&$\leq 3$\\
   &$(1,3,\{1\}^{14})$&$<0.509572431$   &$(s_1,\{1\}^{15})$&$\leq 3$\\
 \hline
 17&$(1,1,2,\{1\}^{14})$&$<0.732770497$  & $(s_1,2,\{1\}^{15})$&$\leq 3$\\
   &$(1,3,\{1\}^{15})$&$<0.546742691$   &$(s_1,\{1\}^{16})$&$\leq 3$\\
 \hline
 18&$(1,2,\{1\}^{16})$&$<0.722868056$  & $(s_1,\{1\}^{17})$&$\leq 3$\\
 \hline
\end{tabular}
\end{center}
\end{table}

\end{document}